\documentclass[leqno,10pt]{article}
\usepackage{amsmath}
\usepackage{array}
\usepackage{enumerate}
\usepackage{graphicx}
\usepackage{amssymb}
\usepackage{amsthm}
\usepackage{xcolor}
\usepackage[version=3]{mhchem}
\usepackage{hyperref}
\input xy
\xyoption{all}

\newtheorem{theorem}{Theorem}
\newtheorem{proposition}{Proposition}
\newtheorem{lemma}{Lemma}
\newtheorem{corollary}{Corollary}

\theoremstyle{definition}
\newtheorem{definition}{Definition}
\newtheorem{remark}{Remark}
\newtheorem{example}{Example}

\newcommand{\R}{\mathbb{R}}

\DeclareMathOperator*{\rank}{rank}
\DeclareMathOperator*{\diag}{diag}

\begin{document}

\title{Tikhonov-Fenichel reduction for parameterized critical manifolds with applications to chemical reaction networks}

\author{Elisenda Feliu\\Department of Mathematical Sciences, University of Copenhagen\\Universitetsparken 5, 2100
Copenhagen, Denmark\\ \tt{efeliu@math.ku.dk}\\
\\
Niclas Kruff\\
Lehrstuhl A f\"ur Mathematik, RWTH Aachen\\
52056 Aachen, Germany\\
\tt{niclas.kruff@matha.rwth-aachen.de}\\
\\
Sebastian Walcher\\
Lehrstuhl A f\"ur Mathematik, RWTH Aachen\\
52056 Aachen, Germany\\
\tt{walcher@matha.rwth-aachen.de}
}


\maketitle
\begin{abstract} We derive a reduction formula for singularly perturbed ordinary differential equations (in the sense of Tikhonov and Fenichel) with a known parameterization of the critical manifold. No a priori assumptions concerning separation of slow and fast variables are made, or necessary.We apply the theoretical results to chemical reaction networks with mass action kinetics admitting slow and fast reactions. For some relevant classes of such systems there exist canonical parameterizations of the variety of stationary points, hence the theory is applicable in a natural manner. In particular we obtain a closed form expression for the reduced system when the fast subsystem admits complex balanced steady states.

\end{abstract}
\section{Introduction}
A fundamental result on singular perturbation reductions, due to Tikhonov and Fenichel, allows to reduce the dimension of an ordinary differential equation with a small positive parameter in the asymptotic limit when the parameter approaches zero. This  theorem has numerous applications in the sciences, {in particular  to chemical and biochemical reaction networks, especially quasi-steady state (QSS) for certain chemical species,} and partial equilibrium approximation (PEA) for slow and fast reactions. However, the application of Tikhonov's and Fenichel's theory to reaction networks may pose some computational problems, and the purpose of the present paper is to address and resolve one of these problems.

Throughout the present work we assume that a suitable small parameter (in a system possibly depending on several parameters) has been identified, hence we deal with a singularly perturbed ordinary differential equation. However, we do not assume the equation to be given in separated fast and slow variables, so the usual version of the reduction theorem is not directly applicable.  In applications, fast-slow variable separation is frequently not satisfied a priori and worse, there may be no explicit way to rewrite the system in fast-slow form. Generally one may circumvent (and to some extent resolve) this problem by resorting to an ``implicit'' version of the reduction, which admits the critical submanifold of phase space as an invariant set, but this approach may also encounter computational feasibility problems. Given this background, we derive in the present paper an explicit singular perturbation reduction that is applicable whenever a parameterization of the critical manifold is known. 

The paper is organized as follows. After some preliminary work (mostly recalling notions and results from the literature) we derive in Section 2 a general formula for Tikhonov-Fenichel reduction when a (possibly local) parameterization of the critical manifold is given, and also consider some special cases. We illustrate the procedure by some small examples, briefly indicating that the range of applications is not restricted to chemical reaction networks. 
However, this reduction formalism seems particularly useful for reaction networks when the partial equilibrium approximation is applicable, since in many instances varieties of stationary points admit a canonical parameterization. This setting is discussed in detail in Section 3, and our results include a closed form reduction formula for fast subsystems that are complex balanced, as well as a discussion of linear attractivity properties of slow manifolds. To finish the paper we discuss some examples.

\section{A reduction formula}\label{sec:reduction}
We consider the singular perturbation reduction of ordinary differential equations, with no a priori assumption on separated (slow and fast) variables. 
Thus let \(U\subseteq \R^{n}\) be open, $\varepsilon_0>0$, and let \(h\) be a smooth function in some neighborhood of $U \times [0,\varepsilon_0)$, with values in $\R^n$. This defines a parameter-dependent system of ordinary differential equations, viz.
\begin{equation}\label{sys}
 \dot x = h^{(0)}(x) + \varepsilon h^{(1)}(x) +\varepsilon^2 \dotso, \quad x\in U, \quad \varepsilon\geq 0,
\end{equation}
and rewritten in slow time scale \(\tau =\varepsilon t\) we have a singularly perturbed system
\begin{equation}\label{sys-slow}
x' = \frac{1}{\varepsilon}h^{(0)}(x) + h^{(1)}(x) +\varepsilon \dotso, \quad x\in U, \quad \varepsilon\geq 0.
\end{equation}
Here \(h^{(0)}\) is called the fast part and \(h^{(1)}\) the slow part of either system. We focus on the behavior of the \eqref{sys}, \eqref{sys-slow} as \(\varepsilon \to 0\), and we will restrict attention to scenarios for which, modulo a coordinate transformation, the classical singular perturbation theorems of Tikhonov \cite{tikh} and Fenichel \cite{fenichel} are applicable. (To be specific, we refer to the version of Tikhonov's theorem as given in Verhulst \cite{verhulst}, Theorem 8.1 ff.; see also \cite{gwmc}, Section 2.1.)  While one can establish intrinsic conditions for the existence of a coordinate transformation to ``Tikhonov standard form'' with separated slow and fast variables, such a transformation cannot generally be obtained in explicit form. We first recall a general implicit reduction procedure developed in \cite{gwmc,nw11}, and then present, as a new result, a version of the reduced system that can be computed explicitly when a parameterization of the critical manifold is known.

\subsection{Review: Tikhonov--Fenichel reduction}
We recall the essential results on coordinate-independent Tikhonov-Fenichel reduction from \cite{gwmc,nw11}; in particular we refer to \cite{gwmc}, Theorem 1 and the subsequent remarks.

\begin{proposition}\label{baseprop} Let system \eqref{sys} be given, and denote by ${\mathcal V}(h^{(0)})$ the zero set of $h^{(0)}$.  Moreover let $0<r<n$ and set $s:=n-r >0$.
\begin{enumerate}[(a)]
\item Assume that $a\in {\mathcal V}(h^{(0)})$ has the following properties.
\begin{itemize}
\item There exists a neighborhood $\widetilde U$ of $a$ such that ${\rm rank}\,Dh^{(0)}(x)=r$ for all $x\in Z:={\mathcal V}(h^{(0)})\cap\widetilde U$; in particular $Z$ is an $s$-dimensional submanifold of $\mathbb R^n$.
\item For all $x\in Z$ there is a direct sum decomposition
\[
 \R^n = {\rm Ker\ } Dh^{(0)}(x) \oplus  {\rm Im\ } Dh^{(0)}(x).
\]
\item For all $x\in Z$ the nonzero eigenvalues of  $Dh^{(0)}(x)$ have real part $<0$.
\end{itemize}
Then in some neighborhood of $a$ there exists an invertible coordinate transformation from \eqref{sys} to Tikhonov standard form
\[
\begin{array}{rcl}
\dot y_1&=& \varepsilon f_1(y_1,y_2) +O(\varepsilon^2)\\
\dot y_2&=& \phantom{\varepsilon} f_2(y_1,y_2) +O(\varepsilon)\\
\end{array}
\]
with separated slow and fast variables; moreover the fast system satisfies a linear stability condition. 
\item Conversely, the conditions in part (a) are necessary for the existence of a local coordinate transformation to Tikhonov standard form.
\item One may choose $\widetilde U$ such that there exists a product decomposition with functions $\mu(x) $ taking values in $\R^{r\times 1}$, $ P(x)$ taking values in $ \R^{n\times r}$, such that 
\begin{equation}\label{zerlegung}
 h^{(0)}(x) = P(x) \mu(x), \quad \text{for all  }x \in Z; 
\end{equation}
moreover \({\rm rank}\ P(a) = r\), ${\rm rank}\ D\mu(a)=r$ and
\[
Z = \mathcal{V}(\mu) \cap \widetilde U.
\]
Here the entries of $\mu$ may be taken as any $r$ entries of $h^{(0)}$ that are functionally independent at $a$.
\item The following system (in slow time) is defined on $\widetilde U$, and admits $Z$ as an invariant set:
\begin{equation}\label{eq:reduktionsformel}
 x ^\prime = \left (I_n - P(x) A(x)^{-1} D\mu(x)\right) h^{(1)}(x),
 \end{equation}
with 
\[
 A(x):= D\mu(x) P(x).
\]
The restriction of this system to $Z$ corresponds to the reduced equation in Tikhonov's theorem.
\end{enumerate}
\end{proposition}
We refer to $Z$ as the local critical manifold (or local asymptotic slow manifold) of system \eqref{sys}.
\begin{remark}\label{rk:zero}
\begin{enumerate}[(a)]
\item Note that $Dh^{(0)}(x)=P(x)D\mu(x)$ on $Z$, due to $\mu(x)=0$. Since $P(x)$ has full rank on $Z$, $Dh^{(0)}(x)$ and $P(x)$ have the same column space.
\item The eigenvalues of $A(x)$, $x\in Z$, are the nonzero eigenvalues of $Dh^{(0)}(x)$ whenever the latter has rank $r$; see  \cite{gwmc}, Remark 3.
\item We call
\begin{equation}\label{eq:reduktionsoperator}
 Q(x):= I_n - P(x) A(x)^{-1} D\mu(x)
 \end{equation}
the {\em projection operator} of the reduction. For each $x$ this is a linear projection of rank $s=n-r$ which sends every element of $\mathbb R^n$ to its kernel component from the kernel-image decomposition with respect to $Dh^{(0)}(x)$. 
\item Formally system \eqref{eq:reduktionsformel} is defined whenever $A(x)$ is invertible, and by Fenichel's results it corresponds to a reduced system as $\varepsilon\to 0$ whenever all eigenvalues of $A(x)$ have nonzero real part (normal hyperbolicity).
\end{enumerate}
\end{remark}
\begin{remark}\label{rk:one} The reduced system may just have the form $x'=0$; in particular this occurs in the following scenario: 
$h^{(0)}$ always admits $n-s$ independent first integrals near any point of $Z$, and locally every point of $Z$ is uniquely determined as an intersection of $Z$ with suitable level sets of these first integrals; see \cite{gwmc}, Subsection 2.3. Now, in the special case when $h^{(1)}$ admits the same first integrals, then $Z$ as well as every intersection of $Z$ with level sets is invariant for the reduced equation, meaning that every point of $Z$ is invariant, thus stationary.  However, the only information to be gained from  $x'=0$ for small $\varepsilon>0$ is that system \eqref{sys-slow} restricted to the invariant manifold has right hand side of order $\varepsilon$ or higher. (Generally the reduced system \eqref{eq:reduktionsformel} in slow time represents only the $O(1)$ term in $\varepsilon$.)
\end{remark}
While Proposition \ref{baseprop} provides a general coordinate-free approach to singular perturbation reduction, the critical manifold $Z$ is given only implicitly via the zeros of $h^{(0)}$, and one cannot generally expect an explicit reduction to a system in $\mathbb R^s$. Moreover there may be {a problem with the feasibility of the computations}, in particular with the computation of the projection matrix $Q$.  Therefore it is natural to search for simplified reduction procedures in special circumstances. One notable scenario appears when a parameterization for the critical manifold is explicitly known, and we will next discuss reduction in this case.

\subsection{Parameterized critical manifolds}\label{sec:parameterized}
We keep the assumptions and notation from Proposition \ref{baseprop}, in particular the decomposition \eqref{zerlegung},
the $s$-dimensional local critical manifold $Z$ (being the zero set of $\mu$, as well as of $h^{(0)}$), and the reduced system 
\begin{equation}\label{shortred}
x^\prime = Q(x)h^{(1)}(x)\quad\text{  on  }Z.
\end{equation}
Now assume that there is an open set $W\subseteq\mathbb R^s$ and a smooth parameterization
\begin{equation}\label{parameq}
\Phi \colon W\to Z, \quad{\rm rank}\,D\Phi(v)=s \text{  for all  }v\in W.
\end{equation}
 Then every solution $x(t)$ of \eqref{shortred} with initial value in $\Phi(W)$ can be written in the form
\[
x(t)=\Phi(v(t)),
\]
for $t$ in some neighborhood of $0$, and differentiation yields
\begin{equation}\label{paramprep}
D\Phi(v(t))\,v^\prime(t)=x^\prime(t)=Q(\Phi(v(t)))\cdot h^{(1)}(\Phi(v(t))).
\end{equation}

The remaining task is to simplify this expression.
\begin{theorem}\label{Parthm}
\begin{enumerate}[(a)]
\item For every $v\in W$ there exists a unique $R(v)\in\mathbb R^{s\times n}$ such that
\[
Q(\Phi(v))=D\Phi(v)\cdot R(v).
\]
\item The reduced system, in parameterized version \eqref{paramprep}, is given by 
\begin{equation}\label{paramend}
v^\prime=R(v)\cdot h^{(1)}(\Phi(v)).
\end{equation}
\item The matrix $R(v)$ is uniquely determined by the conditions
\[
R(v)\cdot P(\Phi(v))=0\quad\text{and  }R(v)\cdot D\Phi(v)=I_s,
\]
and therefore can be obtained from the matrix equation
\[
R(v)\cdot\left(D\Phi(v)\,|\,P(\Phi(v))\right)=\left(I_s\,|\,0\right)
\]
with $\left(D\Phi(v)\,|\,P(\Phi(v))\right)$ invertible. In particular, $v\mapsto R(v)$ is smooth.

\item For every $x\in Z$ let $L(x)\in \R^{s\times n}$ be of full rank $s$ and such that $L(x)Dh^{(0)}(x)=0$; equivalently $L(x)P(x)=0$. Moreover define $L^*(v):=L(\Phi(v))$. Then 
$$ R(v) =  \big(L^*(v)\, D\Phi(v)\big) ^{-1}   L^*(v),$$
and the reduced system, in parameterized form, is given by
\begin{equation}\label{paramendvar}
v' = \big(L^*(v)\, D\Phi(v)\big) ^{-1}   L^*(v)\, h^{(1)}(\Phi(v)).
\end{equation}
\end{enumerate}
\end{theorem}
\begin{proof} 
For every $v\in W$ one has $h^{(0)}(\Phi(v))=0$, and by differentiation
\[
Dh^{(0)}(\Phi(v))D\Phi(v)=0.
\]
Thus the image of $D\Phi(v)$ is contained in the kernel of $Dh^{(0)}(\Phi(v))$, and these two vector spaces have dimension $s$, hence they are equal. In turn, for $x\in Z$ the kernel of $Dh^{(0)}(x)$ is by construction equal to the image of $Q(x)$. Thus, for every $v$ the matrices $Q(\Phi(v))$ and $D\Phi(v)$ have the same column space, and the latter has full rank. Therefore every column of $Q(\Phi(v))$ is a unique linear combination of the columns of $D\Phi(v)$. Rewritten in matrix language, this is the assertion of part (a). Part (b) is now obvious from equation \eqref{paramprep} and injectivity of $D\Phi(v)$.

The first condition given in part (c) is a consequence of part (a), the identity $Q(x)\cdot P(x)=0$ for all $x\in Z$ (which is readily verified from the defining equation \eqref{eq:reduktionsoperator}), and the fact that $D\Phi(v)$ is an injective linear map. The second condition follows from the fact that $D\Phi(v)\cdot R(v)=Q(\Phi(v))$ is a projection of rank $s$, by using Lemma \ref{linalglem} in the Appendix. Invertibility of the matrix  $\left(D\Phi(v)\,|\,P(\Phi(v))\right)$ follows from the direct kernel--image decomposition with respect to $Dh^{(0)}(\Phi(v))$, since the columns of $D\Phi(v)$ span the kernel and the columns of $P(\Phi(v))$ span the image (see \cite{gwmc} for more details).

To prove (d), first
recall from Remark \ref{rk:zero} that $Dh^{(0)}(x)$ and $P(x)$ have the same column space, therefore $L(x)P(x)=0$ on $Z$. This and $R(v) P(\Phi(v))=0$ from part (c) imply that $R(v) = \Lambda(v) L(\Phi(v))$ for all $v\in W$, with $\Lambda(v)\in \R^{s\times s}$ uniquely determined. 
Using now the second condition in part (c), we have
$$  \Lambda(v) L(\Phi(v)) D\Phi(v)=I_s, $$
hence  $ L(\Phi(v)) D\Phi(v)$ is invertible  and
$$ \Lambda(v)=\big(L(\Phi(v)) D\Phi(v) \big)^{-1},   $$
 which leads to the asserted expression. 
\end{proof}
\begin{remark} \label{premark}
\begin{enumerate}[(a)]
\item To determine  the reduced equation \eqref{paramendvar}, there is no need for explicit knowledge of the projection matrix $Q$, or of the matrix $P$ from the decomposition. However, the column space of $Dh^{(0)}(x)$, $x\in Z$,  is a crucial ingredient.
\item On the other hand, knowledge of $P$ and $\mu$ seems indispensable for the computation of $A(x)=D\mu(x)P(x)$, and of $A(\Phi(v))$. Note that the eigenvalues of the latter provide direct information on the stability of the critical manifold; see Remark \ref{rk:zero} (b).
\end{enumerate}
\end{remark}
We consider some special cases in more detail.
\begin{corollary}\label{Parcor}
Assume that 
\[
\Phi(v)=\begin{pmatrix}\Phi_1(v)\\ \Phi_2(v)\end{pmatrix},\quad\text{ with   }\Phi_1(v)\in\mathbb R^s\text{  and } D\Phi_1(v) \text{ invertible, for all }\,v\in W,
\]
and partition 
\[
P(x)=\begin{pmatrix}P_1(x)\\P_2(x)\end{pmatrix}\quad\text{  with  }P_1(x)\in\mathbb R^{s\times r}.
\]
Then 
\[
R(v)=\left(R_1(v)\,|\, R_2(v)\right)
\]
with
\[
\begin{array}{rcl}
R_1(v)&=&  I_s-P_1\left( D\Phi_2 D\Phi_1^{-1}P_1-P_2\right)^{-1} D\Phi_2 D\Phi_1^{-1}\\
R_2(v)&=& P_1\left( D\Phi_2 D\Phi_1^{-1}P_1-P_2\right)^{-1}.
\end{array}
\]
In these expressions the argument of $ D\Phi_1$ and $ D\Phi_2$ is $v$ and the argument of $P_1$ and $P_2$ is $\Phi(v)$.

In the special case when $\Phi_1(v)=v$ we get
\[
\begin{array}{rcl}
R_1(v)&=&  I_s-P_1\left( D\Phi_2 P_1-P_2\right)^{-1} D\Phi_2 \\
R_2(v)&=& P_1\left( D\Phi_2 P_1-P_2\right)^{-1}.
\end{array}
\]
\end{corollary}
\begin{proof}
With $R_i$ given as above one verifies
\[
R_1D\Phi_1+R_2D\Phi_2=I_s\text{   and   } R_1P_1+R_2P_2=0
\]
by direct computation. Rewriting, one obtains
\[
\left(R_1\,|\,R_2\right)\begin{pmatrix} D\Phi_1 & P_1\\
                                                                D\Phi_2 & P_2\end{pmatrix}=\left( I_s\,|\,0\right),
\]
and this is the defining property of $R$ in Theorem \ref{Parthm}.
\end{proof}

\begin{remark}
\begin{enumerate}[(a)]
\item Up to a relabeling of variables in $\mathbb R^n$, a partitioning for $\Phi(v)$ as required in Corollary \ref{Parcor} always exists locally, due to the rank condition on the derivative.
\item The special case $\Phi_1(v)=v$ occurs when the critical manifold is the graph of some function. For this case, reduction formulas were derived earlier by Fenichel \cite{fenichel}, Lemma 5.4, and Stiefenhofer \cite{sti}, Equation (2.13).
\item In the yet more special case that $\Phi_1(v)=v$ and $\Phi_2(v)=0$ the procedure yields the familiar quasi-steady state reduction. Indeed, in this case one has $\mu(x)=x_2$ for $x=(x_1,\,x_2)^{\rm tr}$ and $x_1\in \mathbb R^s$, thus $R_1=I_s$, $R_2=-P_1P_2^{-1}$ and the reduced equation is
\[
v^\prime=\left(I_s\,|-P_1((v,\,0)^{\rm tr})P_2^{-1}((v,\,0)^{\rm tr})\right)\,\begin{pmatrix}h_1^{(1)}((v,\,0)^{\rm tr})\\ h^{(1)}_2((v,\,0)^{\rm tr})\end{pmatrix}.
\]
Ignoring higher order terms in $\varepsilon$ (which are irrelevant for Tikhonov-Fenichel reduction) and renaming variables, one obtains the same system by  setting the second part of
\[
\begin{pmatrix}\dot x_1\\ \dot x_2\end{pmatrix}=\begin{pmatrix} P_1\\ P_2\end{pmatrix}\cdot x_2+\varepsilon \begin{pmatrix}h_1^{(1)}\\ h^{(1)}_2\end{pmatrix}+\cdots
\]
equal to zero, solving for $x_2$, substituting into the first part and passing to slow time. This is another proof of the fact that singular perturbation reduction and QSS reduction agree when the critical manifold is a coordinate subspace. (The first proof was given in \cite{gwz2}, Proposition 5.)
\end{enumerate}
\end{remark}
Finally, with a view on chemical reaction networks, we address conservation laws.
\begin{proposition}\label{conslaw}
Let the smooth real-valued function $\psi$ be defined on some open subset of $U$ which has nonempty intersection with $\Phi(W)$, and assume that $\psi$ is a first integral of system \eqref{sys} for every $\varepsilon$. Then $\widetilde \psi:=\psi\circ \Phi$ is constant or a first integral of the parameterized reduced system \eqref{paramend}.
\end{proposition}
\begin{proof}
By \cite{lawa}, Proposition 8 the restriction of $\psi$  to the critical manifold $Z$ is also a first integral of the reduced system \eqref{eq:reduktionsformel}, thus $D\psi(x)Q(x)h^{(1)}(x)=0$ for all $x\in Z$. Therefore
\[
\begin{array}{rcl}
D\widetilde\psi(v)R(v)h^{(1)}(\Phi(v))&=&D\psi(\Phi(v))D\Phi(v)R(v)h^{(1)}(\Phi(v))\\
 &=&D\psi(\Phi(v))Q(\Phi(v))h^{(1)}(\Phi(v))\\
 &=&0,
\end{array}
\]
which is the characterizing property for first integrals of system \eqref{paramend}.
\end{proof}

\subsection{Illustrative examples}
The following small examples have the primary function to illustrate the arguments and reduction procedures from the previous subsection.
\begin{enumerate}[1.]
\item We consider a (hypothetical) slow-fast system, with fast reaction
\[
 X_1 +  X_2 \rightleftharpoons X_3
\]
and slow reaction
\[
X_1+X_3\rightleftharpoons 2 X_2,
\]
with associated differential equation  (according to the procedure from Subsection~\ref{subsec:reactions} below)
\[
\begin{array}{rcccccccl}
\dot x_1&=&-k_1x_1x_2&+& k_{-1}x_3&-&\varepsilon k_2x_1x_3&+&\varepsilon k_{-2}x_2^2\\
\dot x_2 &=&-k_1x_1x_2&+& k_{-1}x_3&+&2\varepsilon k_2x_1x_3&-&2\varepsilon k_{-2}x_2^2\\
\dot x_3&=&k_1x_1x_2&-& k_{-1}x_3&-&\varepsilon k_2x_1x_3&+&\varepsilon k_{-2}x_2^2.\\
\end{array}
\]
The critical manifold $Z$ is determined by $Kx_1x_2=x_3$, with $K=k_1/k_{-1}$, and we have $P=(1,1,-1)^{\rm tr}$, $\mu(x)=(-k_1x_1x_2+ k_{-1}x_3 )$. A parameterization of $ Z$ is given by
\[
\Phi\colon \mathbb R_{\geq 0}^2\to\mathbb R^3,\quad\begin{pmatrix}v_1\\v_2\end{pmatrix}\mapsto \begin{pmatrix}v_1\\v_2\\Kv_1v_2\end{pmatrix},
\]
hence to determine $R(v)$ via Theorem \ref{Parthm}(c) we have to solve
\[
R(v)\cdot\begin{pmatrix} 1& 0 & 1\\ 0& 1 & 1\\  Kv_2& K v_1& -1\end{pmatrix}=\begin{pmatrix} 1&0&0\\ 0&1&0\end{pmatrix}.
\]
By straightforward calculations one obtains
\[
R(v)=\frac1{1+K(v_1+v_2)}\begin{pmatrix}1+Kv_1&-Kv_1& 1\\
-Kv_2& 1+Kv_2 & 1\end{pmatrix}.
\]
With 
\[
h^{(1)}(\Phi(v))=(-k_2Kv_1^2v_2+ k_{-2}v_2^2)\cdot \begin{pmatrix}1\\-2\\ 1\end{pmatrix}
\]
the reduced system becomes
\[
\begin{pmatrix}v_1^\prime\\ v_2^\prime\end{pmatrix}=\frac{k_2Kv_1^2v_2-k_{-2}v_2^2}{1+K(v_1+v_2)}\begin{pmatrix}-2-3Kv_1\\ 1+3Kv_2\end{pmatrix}.
\]
This is a case where the critical manifold is the graph of the rational function $(x_1,x_2)\mapsto Kx_1x_2$; thus Corollary \ref{Parcor} would also be applicable. Moreover we get
\[
A(\Phi(v))= D\mu(\Phi(v)) P(\Phi(v)) =-\left(k_1(v_1+v_2)+k_{-1}\right)<0 \text{  on } \mathbb R^2_{\geq 0},
\]
hence linear stability of the critical manifold follows by Remark \ref{premark}(b).
\item As a non-hypothetical variant we discuss the system with the same fast reaction as in part 1, but with slow reaction
\[
X_3\rightleftharpoons X_1+X_4, 
\]
and associated differential equation 
\[
\begin{array}{rcccccccl}
\dot x_1&=&-k_1x_1x_2&+& k_{-1}x_3&+&\varepsilon k_2x_3&&\\
\dot x_2 &=&-k_1x_1x_2&+& k_{-1}x_3&&&&\\
\dot x_3&=&k_1x_1x_2&-& k_{-1}x_3&-&\varepsilon k_2x_3& &\\
\end{array}
\]
after discarding the equation for $x_4$. This is Michaelis-Menten with slow degradation of complex to enzyme and product. Here $R(v)$ is the same as in the previous example, and
\[
h^{(1)}(\Phi(v))=k_2Kv_1v_2\cdot \begin{pmatrix}1\\0\\ -1\end{pmatrix}.
\]
The reduced system becomes
\[
\begin{pmatrix}v_1^\prime\\ v_2^\prime\end{pmatrix}=\frac{k_2Kv_1v_2}{1+K(v_1+v_2)}\begin{pmatrix}Kv_1\\ -(1+Kv_2)\end{pmatrix},
\]
and we note a further built-in reduction: The differential equation for the reaction network admits the first integral $\psi=x_1+x_3$ from stoichiometry, hence by Proposition \ref{conslaw} the reduced equation inherits the first integral $\widetilde\psi=v_1+Kv_1v_2$. Thus one ends up with a one dimensional reduced equation, as it should be.
\item For contrast, consider the hypothetical slow-fast system with fast reaction
\[
2 X_1 + 2 X_2 \rightleftharpoons 3 X_3
\]
and the same slow reaction as in part 1.
The differential equation now becomes
\[
\begin{array}{rcccccccl}
\dot x_1&=&-2k_1x_1^2x_2^2&+&2 k_{-1}x_3^3&-&\varepsilon k_2x_1x_3&+&\varepsilon k_{-2}x_2^2\\
\dot x_2 &=&-2k_1x_1^2x_2^2&+&2 k_{-1}x_3^3&+&2\varepsilon k_2x_1x_3&-&2\varepsilon k_{-2}x_2^2\\
\dot x_3&=&3k_1x_1^2x_2^2&-&3 k_{-1}x_3^3&-&\varepsilon k_2x_1x_3&+&\varepsilon k_{-2}x_2^2,\\
\end{array}
\]
and the critical manifold $ Z$ is given by $Kx_1^2x_2^2=x_3^3$, with $K=k_1/k_{-1}$, and $P=(2,2,-3)^{\rm tr}$. A parameterization of $ Z$ is given by
\[
\Phi\colon\,\,\begin{pmatrix}v_1\\v_2\end{pmatrix}\mapsto \begin{pmatrix}v_1^3\\v_2^3\\Kv_1^2v_2^2\end{pmatrix}.
\]
It is obvious that
\[
 L=\begin{pmatrix}1&-1&0\\3&0&2\end{pmatrix}
\]
is of rank two and satisfies $L\cdot P=0$. This yields, by Theorem \ref{Parthm}(d), 
\[
 L\cdot D\Phi(v)=L\cdot \begin{pmatrix}3 v_1^2 & 0\\0& 3v_2^2\\2Kv_1v_2^2&2Kv_1^2v_2\end{pmatrix}=\begin{pmatrix}3v_1^2 & -3v_2^2\\ 9v_1^2+4Kv_1v_2^2 & 4K v_1^2v_2\end{pmatrix}.
\]
Finally the reduced system is given by
\begin{align*}
 v'&=\frac{1}{3v_1v_2(4K(v_1^3+v_2^3)+9v_1v_2)}\begin{pmatrix}4Kv_1^2v_2&3v_2^2\\-9v_1^2-4Kv_1v_2^2&3v_1^2\end{pmatrix}\cdot L\cdot h^{(1)}\\
 &=\frac{1}{3v_1v_2(4K(v_1^3+v_2^3)+9v_1v_2)}\cdot \\
 &\begin{pmatrix}-12K^2v_1^7v_2^3k_2+12Kv_1^2v_2^7k_{-2}-15Kv_1^2v_1^3v_2^4k_2+15v_2^8k_{-2}\\12K^2v_1^6v_2^4k_2+27Kv_1^7v_2^2k_2-12Kv_1v_2^8k_{-2}-15Kv_1^2v_1^5v_2^2k_2-12v_1^2v_2^6k_{-2}\end{pmatrix}.
\end{align*}
One can further reduce the dimension to one by utilizing the first integral $\psi=4x_1+5x_2+6x_3$ from stoichiometry.

In this example we could have chosen a different parameterization 
\[
\Psi\colon \begin{pmatrix}v_1\\v_2\end{pmatrix}\mapsto \begin{pmatrix}v_1\\v_2\\Kv_1^{2/3}v_2^{2/3}\end{pmatrix},
\]
for $ Z$, which directly represents $ Z$ as  the graph of a function, but it may be more convenient to work with a reduced system that has rational right hand side.
\item Finally we sketch an example that is motivated by mechanics, to illustrate that the range of applications does not only include reaction networks.  (See Arnold and Anosov \cite{aran} for background, and also \cite{walchernofo}.) Specifically we look at a pair of coupled nonlinear oscillators
\[
\begin{array}{rcccl}
\dot y_1&=& y_2 &+&\cdots \\
\dot y_2&=&-y_1&+&\cdots\\
\dot y_3&=& \omega y_4&+&\cdots\\
\dot y_4&=&-\omega y_3&+&\cdots
\end{array}
\]
with irrational $\omega>0$,  thus we are in a non-resonant scenario.
 Computing a normal form up to degree three and reduction by invariants $y_1^2+y_2^2$, $y_3^2+y_4^2$ yields a two-dimensional system, which generically allows to decide about stability. But here we look at a degenerate case, with reduced equation
\begin{equation}\label{mechex}
\begin{array}{rcl}
\dot x_1&=&x_1\left(ax_1+bx_2\right)\\
\dot x_2&=&cx_2\left(ax_1+bx_2\right)\\
\end{array}
\end{equation}
and parameters $a<0$, $b>0$ and $c<0$. The choice of signs ensures that the line of stationary points given by $\mu:=ax_1+bx_2=0$ lies in the positive quadrant (which is positively invariant), and also that solutions on the invariant lines $x_1=0$ resp. $x_2=0$ converge to $0$ as $t\to\infty$.

We consider \eqref{mechex} as fast part $h^{(0)}$ of a singularly perturbed system, thus we have
\[
P(x)=\begin{pmatrix}x_1\\ cx_2\end{pmatrix}\text{  and choose  }\Phi(v)=\begin{pmatrix}bv\\ -av\end{pmatrix}, \,v>0.
\]
A straightforward calculation shows that $A(\Phi(v))=ab(1-c)v<0$ whenever $v>0$, so the line $\mu=0$ is attracting in the positive quadrant. Moreover one may choose $L(x)=(cx_2,\,-x_1)$ and thus obtains the reduced equation
\begin{equation}\label{mechexred}
v^\prime =\frac{-1}{ab(1-c)}\cdot\begin{pmatrix}ac & b\end{pmatrix}h^{(1)}(\Phi(v))
\end{equation}
for arbitrary small perturbation $h^{(1)}$. The choice
\[
h^{(1)}(x)=\begin{pmatrix} x_1^3\\ -x_2^4\end{pmatrix}, \quad h^{(1)}(\Phi(v))=\begin{pmatrix} b^3v^3\\ -a^4v^4\end{pmatrix}
\]
is compatible with the mechanical context under consideration here, and yields a positive stationary point for \eqref{mechexred} as well as \eqref{mechex}. For the original system this yields the existence of an invariant torus.
\end{enumerate}


\section{Applications to reaction networks}
While the range of applications of Theorem \ref{Parthm} is not restricted to chemical reaction networks, it is natural to discuss these in greater detail: The consideration of slow and fast reactions leads to critical manifolds that consist of stationary points of a subnetwork, and for some relevant and familiar classes of networks explicit parameterizations of the variety of stationary points exist.  We first recall some general facts about reaction networks and then discuss two special classes. The results will be illustrated by examples.

\subsection{Reaction networks} \label{subsec:reactions}
We briefly recall here the mathematical description of reaction networks according to Feinberg \cite{feinberg}, Horn and Jackson \cite{hoja} (see also the recent monograph \cite{feinbergbook} by Feinberg), and then outline the general setup for networks with fast and slow reactions, as already suggested by some illustrative examples in the previous section.

A \emph{reaction network} on a set of species $\{X_1,\dots,X_n\}$  is a digraph whose nodes are finite linear combinations  of species with nonnegative integer coefficients; each edge is called a reaction. Thus a node is of the form $y=\sum_{i=1}^n a_i X_i$ and is identified with the vector $y=(a_1,\dots,a_n)\in \R^n$.  
We let $x_i$ denote the concentration of $X_i$ and $x=(x_1,\dots,x_n)$. For each reaction (denoted by $y\rightarrow y'$) we assume given a rate function $w_{y\rightarrow y'}(x)\in \R_{\geq 0}$ for $x\in \R^n_{\geq 0}$. This leads to a system of differential equations describing the evolution of the concentrations  in time:
\begin{equation}\label{eq:ode}
\dot{x} = \sum_{\textrm{reactions } y\rightarrow y' } w_{y\rightarrow y'}(x) (y'-y),\qquad x\in \R^n_{\geq 0}.
\end{equation}
Note that $y'-y\in \R^n$ encodes the net production of each species by the occurrence of the reaction $y\rightarrow y'$. The vector subspace spanned by all the $y'-y$ is called the \emph{stoichiometric subspace} of the reaction network.

It is convenient to write the system in matrix-vector form by introducing the \emph{  matrix} $N$ whose columns are the vectors $y'-y$ (after fixing an order of the set of reactions). Then, with $w(x)$ denoting the vector of rate functions in the same order,  \eqref{eq:ode} can be rewritten as
\begin{equation}\label{eq:ode2}
\dot{x} =N w(x),\qquad x\in \R^n_{\geq 0}.
\end{equation}
A frequent choice of rate function is the one from \emph{mass action kinetics}, with
$$ w_{y\rightarrow y'}(x) = k_{y\rightarrow y'} \prod_{i=1}^n x_i^{y_i},$$
with $k_{y\rightarrow y'} >0$ called reaction rate constants and using the convention $0^0=1$. 
For the following we recall some definitions:
\begin{definition}
 Let $x,y\in \R^n$ and $M\in \R^{n\times m}$, with columns $M_1,\ldots,M_m\in \mathbb R^n$.
 \begin{enumerate}[(a)]
\item For $x\in \mathbb R_{>0}^n$ we define 
\[
x^y:=\prod_{i=1}^n x_i^{y_i},
\]
noting that the definition may be extended to all $x\in\mathbb R^n$ when all $y_i$ are nonnegative integers.
  \item For $x\in \mathbb R_{>0}^n$ we define
  \[
   x^{M}:=\begin{pmatrix} x^{M_{_1}}\\ \vdots\\ x^{M_{_m}}\end{pmatrix}\in \mathbb R^m,
  \]
noting that the definition may be extended to all $x\in\mathbb R^n$ when all entries of $M$ are nonnegative integers.
  \item The \textit{Hadamard product} $x\circ y$ is defined as the componentwise multiplication of the two vectors $x,y$, i.e.\
  \[
   x\circ y=\begin{pmatrix}x_1\cdot y_1\\ \vdots\\ x_n\cdot y_n\end{pmatrix}.
  \]
 \end{enumerate}
\end{definition}

In view of the last definition we can rewrite the reaction network for mass action kinetics  in the form
\begin{equation}\label{eq:ode3}
\dot x=N\cdot w(x)=N\cdot (K\circ x^{Y}),
\end{equation}
where $K\in \R_{> 0}^{m}$ ($m$ is the number of reactions) is a vector containing the reaction rate constants and $Y\in \R^{n\times m}$ is  the matrix whose columns are the reactant vectors of each reaction, called the kinetic order matrix.

It follows directly from \eqref{eq:ode2} that any vector in the left-kernel of $N$ defines a linear first integral, regardless of the form of $w(x)$. 
These linear first integrals are commonly referred to as conservation laws, and their common level sets are called stoichiometric compatibility classes. 
If each connected component of the reaction network has exactly one terminal strongly connected component, then all linear first integrals of \eqref{eq:ode2} arise in this way; see Feinberg and Horn \cite{feinberg-invariant}.  Finally we recall the notion of {\em deficiency} of the reaction network, which is defined as the number of nodes minus the rank of $N$ minus the number of connected components; {see e.g. Horn \cite{horn} or Feinberg \cite{feinberg:deficiency}.}

\medskip
We turn now to a scenario with prescribed slow and fast reactions; see also the discussions in Heinrich and Schauer \cite{hesch}, Lee and Othmer \cite{lo}. The subdigraph induced by the fast reactions is itself a reaction network with the same set of species, which we call the \emph{fast subnetwork}.   We stipulate that even if some species are not part of any fast reaction, we still consider them as part of the fast subnetwork. 
We have a corresponding stoichiometric matrix $N_{\mathsf f}$ and rate vector $w_{\mathsf f}(x)$, such that, in the notation of Section~\ref{sec:reduction},
\begin{equation}\label{eq:h0:fast}
 h^{(0)}(x)= N_{\mathsf f}w_{\mathsf f}(x)=N_{\mathsf f}\cdot (K_{\mathsf f}\circ x^{Y_{\mathsf f}}).
\end{equation}
Analogously, we have 
\begin{equation}\label{eq:h1:slow}
 h^{(1)}(x)= N_{\mathsf s}w_{\mathsf s}(x)=N_{\mathsf s}\cdot (K_{\mathsf s}\circ x^{Y_{\mathsf s}})
\end{equation}
for the slow subsystem.
Keeping the  notation from Section \ref{sec:reduction}, we let $Z$ be the zero set of $ h^{(0)}$ (possibly restricted to a neighborhood $\widetilde U$ of some point), and let $r$ denote the rank of $D h^{(0)}(x)$, $x\in Z$.
Then clearly $\rank N_{\mathsf f}\geq r$, but the inequality may be strict; see Heinrich and Schauer \cite{hesch} and also Section 3 of \cite{gwmc}. In the present paper we will, however, restrict attention to the case when equality holds:\\

\noindent{\bf Blanket hypothesis.} We impose on system \eqref{eq:h0:fast} the conditions of Proposition \ref{baseprop}(a)  and the additional condition that $\rank N_{\mathsf f}={\rm rank}\,Dh^{(0)}(x)=r$, $x\in Z$.\\

Due to nonnegativity of concentrations, the points of $Z$ will be in $\mathbb R^n_{\geq 0}$. In some instances we will require that the neighborhood $\widetilde{U}$ in Proposition~\ref{baseprop} is even a subset of  $\R^n_{>0}$, and likewise we will occasionally require that the domain $W$ of the parameterization  $\Phi$ is a subset of $\R^s_{>0}$.
By our assumption the zero set $Z$ of $h^{(0)}$ has dimension $s=n-r$. 
Assume now that there exists a smooth parameterization 
$$ \Phi\colon W \rightarrow Z$$
with $\rank D\Phi(v)=s$ on $W$. 

\begin{proposition}\label{prop:R} Let system  \eqref{eq:h0:fast} be given, with a parameterization $\Phi$ of the critical manifold as in \eqref{parameq}, and assume the blanket hypothesis holds on $\Phi(W)$.
Let $L_{\mathsf f}\in \R^{s\times n}$ be a matrix whose rows form a basis of the left-kernel of $N_{\mathsf f}$. 
Then the matrix $R(v)$ in  Theorem~\ref{Parthm} is given as
$$ R(v) =  \big(L_{\mathsf f}\, D\Phi(v)\big) ^{-1}   L_{\mathsf f},$$
and the reduced system is
$$v' = \big(L_{\mathsf f}\, D\Phi(v)\big) ^{-1}   L_{\mathsf f} \, h^{(1)}(\Phi(v)), \qquad v\in W.$$
Furthermore, the column space of the matrix $P(x)$ in any decomposition of $h^{(0)}(x)$ as in Proposition~\ref{baseprop}(c) equals the column space  of $N_{\mathsf f}$. 
\end{proposition}
\begin{proof} 
This follows from Theorem~\ref{Parthm}(d), since $L_{\mathsf f}Dh^{(0)}(x)=L_{\mathsf f} N_{\mathsf f} Dw_{\mathsf f}(x) = 0$ on $Z$.
\end{proof}

Note that by Remark~\ref{rk:one}, if $\rank N = \rank N_{\mathsf f}$, then the reduced system is given by
$ v' =0.$

\subsection{Canonical parameterizations for some classes}
To find a function $\Phi(v)$ which yields a parameterization of positive steady states of the fast subnetwork, several strategies can be employed. We review here the two most common approaches.
Throughout we use the blanket hypothesis, denoting the rank of the stoichiometric matrix $N_{\mathsf f}$ by 
$r$, the number of species of the full network by $n$, and let $s=n-r$.
For simplicity, we consider mass action kinetics, although several results hold for more general classes of rate functions.


\subsubsection{Non-interacting sets and rational parameterizations }\label{subnoninter}
As was shown in \cite{fwsiam}, non-interacting sets of species may be utilized to find rational parameterizations of the steady states, given certain conditions.
Thus consider a subset of species $\mathcal{Y}=\{X_{i_1},\dots,X_{i_r}\}$, with the following assumptions. 
\begin{enumerate}[(i)]
\item For every fast reaction $y\rightarrow y'$, both the sum of the coefficients of the species in $\mathcal{Y}$ in $y$ and the sum of the coefficients in $y'$ is at most one. This means that no pair of species in $\mathcal{Y}$ appear together at one side of a reaction, and further no species appears with coefficient greater than $1$. 
\item The rank of the submatrix of  $N_{\mathsf f}$ given by the rows $i_1,\dots,i_r$ is equal to $r$. 
\item Consider the network induced by the fast subnetwork by setting all species not in $\mathcal{Y}$ to zero. 
For each species $X_{i_j}$ in $\mathcal{Y}$, there is a directed path from $X_{i_j}$ to $0$ in this induced network.
\end{enumerate}
In the nomenclature of \cite{fwsiam}, assumption (i) means that $\mathcal{Y}$ is non-interacting, (ii) means that no conservation law has support in $\mathcal{Y}$, and (iii) means that there exists a spanning tree rooted at $*$ in the appropriate digraph (see \cite{fwsiam}, Section 8 for details). 

Let $X_{\ell_1},\dots,X_{\ell_s}$ be the  species not in $\mathcal{Y}$. 
If $\mathcal{Y}$ satisfies (i), (ii) and (iii), then
the components $i_1,\dots,i_r$ of $h^{(0)}(x)$ form a linear system in $x_{i_1},\dots,x_{i_r}$  that has a unique solution in terms of   $x_{\ell_1},\dots,x_{\ell_s}$.
Furthermore, the solution is a rational function in $x_{\ell_1},\dots,x_{\ell_s}$ and in the reaction rate constants $k_{y\rightarrow y'}>0$,  with all coefficients positive  \cite{fwsiam}. The solution can be found using graphical procedures, but in practice, solving the system of linear equations is the preferred approach (see  \cite{fwsiam} for more on this). 

By this procedure one obtains a parameterization of the  zero set $Z$ of $h^{(0)}$ in $s$ variables $v_i= x_{\ell_i}$, $i=1,\dots,s$.
Further, clearly $\rank D\Phi(v)=s$. In \cite{flww} some conditions are stated which guarantee that the assumptions in Proposition~\ref{baseprop}(a) are satisfied.

\subsubsection{Monomial parameterizations and deficiency zero networks }
We next consider another common scenario occurring, for instance, for so-called complex balanced steady states (see  Feinberg \cite{feinberg}, Horn and Jackson \cite{hoja}) and networks with toric steady states (see P\'erez Mill\'an et al. \cite{PDSC}, M\"uller et al. \cite{MFRCSD}). In this scenario the zero set $Z$ of  $h^{(0)}$ in $\R^n_{>0}$ agrees with the solution set of a collection of binomial equations
\begin{equation}\label{eq:binomial}
a_\ell(k) x^{u_\ell} - b_\ell(k) x^{{c}_\ell}=0,\qquad x\in \R^n_{>0}, \ \ell=1,\dots,{q},
\end{equation}
where $u_\ell,{c}_\ell\in \R^n$ and $a_\ell(k),b_\ell(k)$ are  polynomials in the parameters of the rate functions that only attain positive values for valid $k$. Here, for the sake of simplicity, we restrict attention to the case when all $x_i>0$.
Under these assumptions, the solution set $Z$ to \eqref{eq:binomial} equals the solution set of
\begin{equation}\label{eq:monomial}
 x^{u_\ell-{c}_\ell} = \frac{b_\ell(k)}{a_\ell(k)},\qquad x\in \R^n_{>0},  \ \ell=1,\dots,{q}.
\end{equation}
The solution set of \eqref{eq:monomial}, if non-empty, admits a monomial parameterization of the following form. Let $x^*$ be any fixed solution of \eqref{eq:monomial} and $M\in \R^{n\times {q}}$ the matrix whose columns are $u_\ell-{c}_\ell$.
If $x$ is a solution to \eqref{eq:monomial} then
$$ x^M =  (x^*)^M.   $$
It is a classical result (see for example Lemma 3.7 in M\"uller et al. \cite{MFRCSD}) that the solution set to this equation, and hence $Z$, can be parameterized in the form
\begin{equation}\label{binompar}
\Phi(v)=x^*\circ v^B = ( x^*_i v^{b_i})_{i=1,\dots,n}, \qquad v\in \R^d_{>0}.
\end{equation}
Here  $d=\dim \ker M^{\rm tr}$, and $b_1,\dots,b_n$ are the columns of a matrix $B\in \R^{d\times n}$ with row span equal to $\ker M^{\rm tr}$ and $\ker B^{\rm tr}=\{0\}$ (thus $d={\rm rank}\,B$). 
With an easy computation one verifies the well-known identity
\begin{equation}
D \Phi(v) = \diag(x^* \circ v^B) B^{\rm tr} \diag(1/v), 
\end{equation}
where, for a vector $\alpha$, $\diag(\alpha)$ denotes the diagonal matrix with the entries of $\alpha$ in the diagonal, and $1/v$ is  defined component-wise.
By this identity, the rank of $D \Phi(v)$ equals $d$, the rank of $B$, and therefore we are in the setting of Subsection~\ref{sec:parameterized} provided that $d=s$.
With this in place, by Proposition \ref{prop:R} the matrix $R(v)$ in Theorem~\ref{Parthm} becomes
\begin{equation}
\begin{array}{rcl}
R(v) &=&  \big(L_{\mathsf f}\,  \diag(x^* \circ v^B) B^{\rm tr} \diag(1/v) \big) ^{-1}   L_{\mathsf f} \\
&=& \diag(v)   \big(L_{\mathsf f}\,  \diag(x^* \circ v^B) B^{\rm tr}  \big) ^{-1}L_{\mathsf f}.
\end{array}
\end{equation}
We turn now to the special case of \emph{complex-balanced steady states} for mass action kinetics. These are steady states such that for each fixed node $y$ of the fast subnetwork, it holds that
\begin{equation}\label{eq:CB}
 \sum_{\textrm{reaction } y' \rightarrow y} w_{y'\rightarrow y}(x) (y-y')  = \sum_{\textrm{reaction } y \rightarrow y'} w_{y\rightarrow y'}(x) (y'-y).
\end{equation}
As shown in Feinberg \cite{feinberg} and Horn \cite{horn}, a necessary condition for complex balanced steady states to exist is that each connected component of the fast subnetwork is strongly connected; this is what is known as a weakly reversible reaction network.  
In this case, if the parameters $k_{y\rightarrow y'}$ satisfy certain algebraic conditions, then there are positive complex balanced steady states and any positive steady state is complex balanced. 
Furthermore, if the deficiency of the fast subnetwork is zero, then any positive steady state is complex balanced, independent of the values of the parameters.

The set $Z$ of positive  complex balanced steady states  agrees with the solution set of a collection of binomial equations of the form
$$K_{ij} x^{y_i}  - K_{ji} x^{y_j} =0,    \quad  x\in \R^n_{>0},$$
for every pair of nodes $y_i,y_j$ in the same connected component of the reaction network, and with $K_{ij}$ and $K_{ji}$ positive polynomials in the parameters $k_{y\rightarrow y'}$ for the reactions in the same connected component \cite{CDSS}.  
This implies that the column span of the matrix $M$ above agrees with the column span of $N_{\mathsf f}$, and therefore $\ker M^{\rm tr}$ has rank $s$. As a suitable matrix $B$ one can choose $L_{\mathsf f}$ as in Proposition \ref{prop:R}. We thus obtain an explicit expression for the reduced system.
\begin{proposition}\label{closedformularreduction}
Assume that the fast subsystem \eqref{eq:h0:fast} has a positive complex balanced steady state $x^*$.
Then:
\begin{enumerate}[(a)]
\item With the notation of Theorem \ref{Parthm} and Proposition \ref{prop:R}, there is a parameterization \eqref{binompar} of the critical manifold with $B=L_{\mathsf f}$, parameter space $\mathbb R_{>0}^s$ and
\[
R(v) = 
 \diag(v)   \big(L_{\mathsf f}\,  \diag(x^* \circ v^{L_{\mathsf f}}) L_{\mathsf f}^{\rm tr}  \big) ^{-1}L_{\mathsf f}.
\]
\item The reduced system is given by:  
\begin{equation}\label{combalred}
\begin{array}{rcl}
 v'&=&R(v)\cdot h^{(1)}(\Phi(v))=R(v)\cdot N_{\mathsf s}\cdot \left(K_s\circ (x^*)^{Y_{\mathsf s}}\circ v^{L_{\mathsf f}\cdot Y_{\mathsf s}}\right)\\
 &=&\diag(v)   \big(L_{\mathsf f}\,  \diag(x^* \circ v^{L_{\mathsf f}}) L_{\mathsf f}^{\rm tr}  \big) ^{-1}L_{\mathsf f}\cdot N_{\mathsf s}\cdot \left(K_{\mathsf s}\circ (x^*)^{Y_{\mathsf s}}\circ v^{L_{\mathsf f}\cdot Y_{\mathsf s}}\right).
\end{array}
\end{equation}
\end{enumerate}
\end{proposition}
\begin{proof}
Part (a) is clear, while part (b) follows immediately with part (a) and 
  \begin{align*}
h^{(1)}(\Phi(v))&=h^{(1)}\left(x^*\circ v^{L_{\mathsf f}}\right)=N_{\mathsf s}\cdot \left(K_{\mathsf s}\circ \left(x^*\circ v^{L_{\mathsf f}}\right)^{Y_{\mathsf s}}\right)\\
&=N_{\mathsf s}\cdot \left(K_{\mathsf s}\circ (x^*)^{Y_{\mathsf s}}\circ (v^{L_{\mathsf f}})^{Y_{\mathsf s}}\right)\\
&=N_{\mathsf s}\cdot \left(K_{\mathsf s}\circ (x^*)^{Y_{\mathsf s}}\circ v^{L_{\mathsf f}\cdot Y_{\mathsf s}}\right).
  \end{align*}
\end{proof}
\subsubsection{Attractivity of the critical manifold}

In the discussion so far we restricted attention to computing a reduced system on the critical manifold and did not address the question whether the attractivity condition from Proposition \ref{baseprop}(a) is satisfied. Of course, Remarks \ref{rk:zero} and \ref{premark} are available, but due to our consideration of slow and fast reaction networks one also may resort to known properties of certain classes of reaction networks. There are general attractivity results available for complex balanced {\em mechanisms} as introduced by Horn \cite{horn}, which are of primary interest in our setting, since we only require positivity of reaction constants in our considerations. By Horn \cite{horn}, Theorem 4A a mechanism is complex balanced for all choices of reaction rate constants if and only if it is weakly reversible and has deficiency zero. For these systems Feinberg \cite{feinberg} proved in Remark C.2 that every steady state is linearly attractive (within its stoichiometric compatibility class). We therefore have:
\begin{proposition}\label{stabzprop}
  If the fast subsystem \eqref{eq:h0:fast} is weakly reversible and of deficiency zero, then all non-zero eigenvalues of the Jacobian have negative real part, hence {all} the conditions of Proposition \ref{baseprop}(a) hold. 
\end{proposition}

\subsubsection{Examples}

\begin{example}
We consider the following reaction network with mass action kinetics, where the numbers $k_{y\rightarrow y'}$ are written as labels of the reactions:
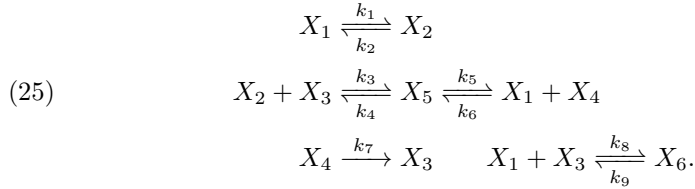
\begin{align}
X_1  & \ce{<=>[k_1][k_2]} X_2 \nonumber \\
 X_2+X_3 & \ce{<=>[k_3][k_4]} X_5  \ce{<=>[k_5][k_6]} X_1+X_4  \label{ex:twocomp}\\
 X_4 & \ce{->[k_7] }X_3 \qquad X_1+X_3 \ce{<=>[k_8][k_9]} X_6. \nonumber
\end{align}
This network is a two-component system where $X_1,X_2$ are the unphosphorylated and phosphorylated forms of the histidine kinase and $X_3,X_4$ are the unphosphorylated and phosphorylated forms of the response regulator. Further we have a dead-end complex between the unphosphorylated forms of both proteins. 

We now look at the slow-fast scenario where the fast reactions are those with labels $k_1,\dots,k_6$, such that  the fast subnetwork is
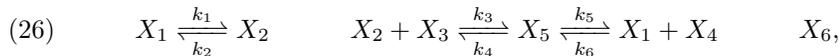
\begin{align}
X_1  & \ce{<=>[k_1][k_2]} X_2&
 X_2+X_3 & \ce{<=>[k_3][k_4]} X_5  \ce{<=>[k_5][k_6]} X_1+X_4   &   X_6 , \label{eq:fast}
\end{align}
and the slow reactions  have labels $k_7,k_8,k_9$.  With this choice of slow-fast reactions, we have
\[
 N_{\mathsf f}=\begin{pmatrix}-1 & 1 & 0 & 0 & 1 & -1\\ 1 & -1 & -1 & 1 & 0 & 0\\ 0 & 0 & -1 & 1 & 0 & 0\\ 0 & 0 & 0 & 0 & 1 & -1\\ 0 & 0 & 1 & -1 & -1 & 1\\ 0 & 0 & 0 & 0 & 0 & 0\end{pmatrix},\qquad N_{\mathsf s}=\begin{pmatrix}0 & -1 & 1\\ 0 & 0 & 0\\ 1 & -1 & 1\\ -1 & 0 & 0\\ 0 & 0 & 0\\ 0 & 1 & -1\end{pmatrix}
\]
and
\[
Y_{\mathsf s}=\begin{pmatrix} 0 & 1 & 0 \\ 0 & 0 & 0 \\ 0 & 1 & 0 \\ 1 & 0 & 0 \\ 0 & 0 & 0 \\ 0 & 0 & 1\end{pmatrix},\quad K_{\mathsf s}=\begin{pmatrix}k_7\\k_8\\k_9\end{pmatrix}.
\]
The fast reaction network in \eqref{eq:fast} has $6$ nodes  and three connected components, and the rank of $N_{\mathsf f}$ is $r=3$ (hence also $s=3$). Therefore the deficiency is zero and any positive steady state, that is, any element of $Z$, is complex balanced and a solution to a set of binomial equations.
Under mass action, the steady states of this fast subnetwork are the solutions to 
\begin{align}
-k_3x_2x_3+k_1x_1-k_2x_2+k_{{4}}x_5 &=0, \nonumber \\ 
-k_6 x_{{1}}x_{{4}}+k_{{5}}x_{{5}}&=0,  \label{eq:TC:ss} \\
k_{{3}}x_{{2}}x_{{3}}+k_{{6}}x_{{4}}x_{{1}}-k_{{4}}x_{{5}}-k_{{5}}x_{{5}}&=0,\nonumber
\end{align} 
and we easily verify that 
\[
x^*=\big(1,\tfrac{k_1}{k_2},  \tfrac{k_2k_4k_6}{k_1k_3k_5} , 1,    \tfrac{k_6}{k_5},1\big)^{\rm tr}
\]
 is a positive steady state of the fast subnetwork. 
With Proposition \ref{closedformularreduction}(b) the reduced system can be computed. We choose
\begin{equation}\label{eq:Lf}
L_{\mathsf f}= \begin{pmatrix}
1 & 1 & 0 & 0 & 1 & 0 \\
0 & 0 & 1 & 1 & 1 & 0 \\ 
0 & 0 & 0 & 0 & 0 & 1
\end{pmatrix}
\end{equation}
and obtain the following parameterization of $Z$:
\[
 \Phi(v)= x^* \circ \begin{pmatrix}v_1\\ v_2\\ v_3\end{pmatrix}^{\hspace{-3pt}L_{\mathsf f}} = x^* \circ \begin{pmatrix}v_1\\ v_1\\ v_2\\ v_2 \\ v_1v_2\\ v_3\end{pmatrix}=
  \begin{pmatrix}v_1\\[4pt] \tfrac{k_1}{k_2}v_1\\[4pt]  \tfrac{k_2k_4k_6}{k_1k_3k_5} v_2\\[4pt] v_2\\ \tfrac{k_6}{k_5}v_1v_2\\[4pt] v_3\end{pmatrix}.
\]
Using equation \eqref{combalred}, the reduced system is found:
\begin{align*}
 v'=\frac{1}{\xi}\cdot \begin{pmatrix} \tfrac{1}{k_1k_3}k_2 (k_1 k_3 k_5+k_2 k_4 k_6) (-k_2 k_4 k_6k_8 v_1v_2 + k_1 k_3 k_5 k_9 v_3)\\
 k_5 (k_1+k_2) (-k_2 k_4 k_6k_8 v_1v_2 + k_1 k_3 k_5 k_9 v_3) \\
\tfrac{\xi}{k_1k_3k_5}  \cdot   (k_2 k_4 k_6k_8 v_1v_2 - k_1 k_3 k_5 k_9 v_3) \end{pmatrix},
\end{align*}
where $\xi$ is given by 
\begin{align*}
\xi=& k_1( k_1  k_3  k_5 +   k_2  k_3  k_5)  k_6 v_1+ k_2(  k_1    k_3  k_5  +   k_2  k_4  k_6)k_6 v_2\\
     &+ k_5(k_1+k_2) (k_1 k_3 k_5+k_2 k_4 k_6).
\end{align*}

In addition, we conclude by Proposition \ref{stabzprop} that  all non-zero eigenvalues of $Dh^{(0)}$ have negative real part on $Z$.

Observe that the parameterization $\Phi(v)$ is not unique. For example, choosing another starting steady state
$$x^*=\Big(1, \frac{k_1}{k_2},1,\frac {k_1k_3k_5  }{k_2k_4k_6},\frac{k_1k_3 }{k_2k_4},1\Big)^{\rm tr},$$
we obtain the parameterization
$$\Phi(v)=\Big( v_1,  \tfrac {k_1v_1}{k_2}, v_2 , \tfrac {k_1k_3k_5 v_2 }{k_2k_4k_6},\tfrac {k_1k_3 v_1 v_2 }{k_2k_4} ,v_3\Big)^{\rm tr},$$
and the reduced system
\begin{align*}
v_1' &= -\tfrac{k_2k_4}{q(v)}  (k_{1}k_{3}k_{5}+k_{2}k_{4}k_{6})(k_8v_1v_2 - k_9v_3) \\
v_2' &= -\tfrac{k_2k_4}{q(v)}k_4k_6  (k_{1}+k_{2})(k_8v_1v_2 - k_9v_3)  \\
v_3'&= k_8v_1v_2 - k_9v_3,
\end{align*}
with 
\begin{align*}
q(v) & =k_{1}k_{3} (k_{1}k_{3}k_{5}+k_{2}k_{4}k_{6})v_{2} +
k_1k_3k_4k_6(k_{1}+k_{2})v_{1} \\ & +
k_4(k_1+k_2)(k_{1}k_{3}k_{5}+k_{2}k_{4}k_{6}).
\end{align*}
The same parameterization is obtained by eliminating $x_2,x_4,x_5$  after realizing 
 that the set of species $\{X_2,X_4,X_5\}$ satisfies (i), (ii) and (iii) in Subsection~\ref{subnoninter}. 
 \end{example}

\begin{example} 
If we remove the reactions with label $k_8,k_9$ from the network \eqref{ex:twocomp}, then the stoichiometric matrices of both the fast subnetwork and the full network have rank $3$. Hence, if the reduction with a parameterized critical manifold is possible, the reduced system is $v'=0$. 
\end{example}

\begin{example}
We analyse the reaction network
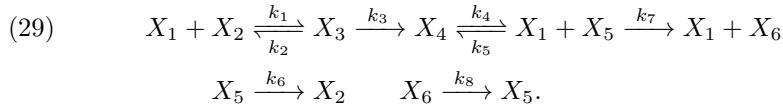
\begin{align}
 X_1+X_2  & \ce{<=>[k_1][k_2]} X_3 \ce{->[k_3]} X_4  \ce{<=>[k_4][k_5]} X_1+X_5  \ce{->[k_7]} X_1+X_6  \label{ex:non_monomial}\\
X_5 & \ce{->[k_6]} X_2\qquad  X_6  \ce{->[k_8] }X_5. \nonumber
\end{align}
This system can be interpreted as a dual phosphorylation cycle with $X_1$ the kinase catalysing the phosphorylation of a substrate $S$ with two phosphorylation sites. Then $X_2$, $X_5$, $X_6$ correspond to the phosphoforms with no, one, two phosphate groups respectively, and $X_3$ and $X_4$ are intermediate enzyme-substrate forms. Dephosphorylation proceeds without a phosphatase. We let the fast system be all reactions involved in the conversion $X_2 \leftrightarrow X_5$, namely those with label $k_1,\dots,k_6$.
Hence the reactions with label $k_7,k_8$ are slow. With this choice, we have
\begin{align*}
 N_{\mathsf f}=\begin{pmatrix}-1&1&0&1&-1&0\\-1&1&0&0&0&1\\1&-1&-1&0&0&0\\0&0&1&-1&1&0\\0&0&0&1&-1&-1\\0&0&0&0&0&0\end{pmatrix},
\qquad 
 L_{\mathsf f}=\begin{pmatrix}0&0&0&0&0&1\\-1&1&0&0&1&0\\1&0&1&1&0&0\end{pmatrix},
\end{align*}
and 
\begin{align*}
h^{(0)}(x) &= \big( -k_1x_1x_2-k_5x_1x_5+k_2x_3+k_4x_4, -k_1x_1x_2+k_2x_3+k_6x_5, \\ & \qquad k_1x_1x_2-k_2x_3-k_3x_3, k_5x_1x_5+k_3x_3-k_4x_4, \\ & \qquad -k_5x_1x_5+k_4x_4-k_6x_5, 0 \big)^{\rm tr} \\
 h^{(1)}(x) & =\big( 0, 0, 0, 0, -k_{7}x_1x_5+k_8x_6, k_{7}x_1x_5 - k_8x_6\big)^{\rm tr}.
\end{align*}

The fast network has deficiency $1$, since the rank of $N_{\mathsf f}$ is $3$, and the network has 6 nodes and two connected components. Thus, the steady states are not complex balanced for all $k$. Instead, we observe that the set
$\{X_3,X_4,X_5\}$ satisfies (i)-(iii) in Subsection~\ref{subnoninter}. Indeed, (i) and (ii) are easy to check. For (iii), 
the induced network obtained after setting the species not in this set to zero is  
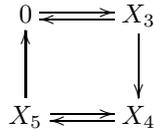
\begin{center}
\begin{minipage}[h]{0.4\textwidth}
\xymatrix{
 0   \ar@<0.3ex>[r] & X_3  \ar@<0.3ex>[l] \ar[d] \\
 X_5   \ar[u] \ar@<0.3ex>[r]      & X_4  \ar@<0.3ex>[l]   }
 \end{minipage}
\end{center}%
and clearly there is directed path to $0$ from every species. 
This implies that $x_3,x_4,x_5$ can be solved from the system $h^{(0)}(x)_{3,4,5}=0$ to obtain the following parameterization  (where $v_1=x_1,v_2=x_2,v_3=x_6$)
\begin{align*}
 \Phi \colon (v_1,v_2,v_3)\mapsto \begin{pmatrix}v_1\\v_2\\\frac{k_1}{k_2+k_3}v_1v_2\\\frac{k_1k_3}{k_4k_6(k_2+k_3)}(k_5v_1+k_6)v_1v_2\\\frac{k_1k_3}{k_6(k_2+k_3)}v_1v_2\\v_3\end{pmatrix},
\end{align*}
which has an entry that is not monomial. Next we compute the matrix $R(v)$ and the reduced system using Proposition~\ref{prop:R}. We   have 
\begin{equation*}
 D\Phi(v)= \begin{pmatrix}
1&0&0\\ 0&1&0
\\ \frac{k_1v_2}{k_2+k_3} & \frac{k_1v_1}{k_2+k_3}  &0 \\[4pt]
\frac{ k_1k_3 ( 2 k_5v_1+k_6) v_2 }{k_4k_6( k_2+k_3 )}  & \frac{( k_5 v_1+k_6 ) k_1k_3v_1}{k_4k_6( k_2+k_3 )  } &0\\[4pt] \frac {k_1k_3v_2}{k_6(k_2+k_3) }  & \frac {k_1k_3v_1}{k_6(k_2+k_3) }  &0 \\
0&0&1 
 \end{pmatrix},
\end{equation*}
and using $ R(v_1,v_2,v_3) = \big(L_{\mathsf f}\, D\Phi(v)\big) ^{-1}   L_{\mathsf f}$
 the reduced system  is given by
\begin{align*}
 \begin{pmatrix} v'_1\\  v'_2\\  v'_3\end{pmatrix}=
 \frac{\xi_1}{\xi_2}
 \begin{pmatrix} 
 k_{1}  (k_{3} k_{5} v_{1}+k_{3} k_{6}+k_{4} k_{6}) v_{1}\\
-(2 k_{1} k_{3} k_{5} v_{1} v_{2}+k_{1} k_{3} k_{6} v_{2}+k_{1} k_{4} k_{6} v_{2}+k_{2} k_{4} k_{6}+k_{3} k_{4} k_{6})\\
 \frac{\xi_{2}}{ k_{6} (k_{3}+k_{2})},
 \end{pmatrix}
\end{align*}
where
\begin{align*}
\xi_1 &= k_{1} k_{3} k_{7} v_{1}^2 v_{2}- (k_{2}+k_3) k_{6} k_{8} v_{3}  \\
 \xi_2&= k_{1}^2  k_{3}^2  k_{5} v_{1}^2 v_{2} + (k_2+k_3) \Big(
   k_{1}  k_3k_5 k_6 v_{1}^2
  +2  k_{1}    k_{3}  k_{5}  k_{6}   v_{1} v_{2}
  \\
     & \quad + k_{1} k_6 (k_3k_4+k_3k_6+k_4k_6) v_{1}
    + k_{1}  (k_3+k_4)k_{6}^2    v_{2}
     +  k_4(k_2+k_3)k_6^2\Big).
\end{align*}

We next verify that the eigenvalue condition in Proposition~\ref{baseprop}(a) is satisfied. To this end we check the eigenvalues of the matrix $A(x)=D\mu(x)\cdot P(x)$ with 
$$ P= \begin{pmatrix} 1 & 0 & 0 \\ 0 & 1 & 0 \\ -1 & 0 & -1 \\ 0 & 0 & 1 \\ 1 & -1 & 0 \\0 & 0 & 0 \end{pmatrix},\qquad 
\mu(x)= \begin{pmatrix}    -k_1x_1x_2-k_5x_1x_5+k_2x_3+k_4x_4 \\  -k_1x_1x_2+k_2x_3+k_6x_5 \\  k_5x_1x_5+k_3x_3-k_4x_4 \end{pmatrix}.$$
Using the Routh-Hurwitz conditions (see Gantmacher \cite{gant}, Ch.~V, \S 6) for its characteristic polynomial $\chi_A(\lambda)=\lambda^3+\sigma_1\lambda^2+\sigma_2\lambda+\sigma_3$ of degree $3$, we obtain:
\begin{align*}
\sigma_1=&  k_1x_1+k_1x_2+k_5x_1+k_5x_5 + k_2+k_3+k_4+k_6 \\
\sigma_2=&k_{1}k_{5} x_1(x_{1}+ x_{2}+ x_{5})+ ( k_{1}k_{3}+k_{1}k_{4}+k_{1}k_{6}+k
_{2}k_{5}+k_{3}k_{5} ) x_{1} \\ & \quad +k_{1} ( k_{3}+k_{4}+k_{6} ) x_{2}+k_{5} ( k_{2}+k_{3}+k_{6} ) x_{5}+
(k_{2}+k_3)(k_{4}+k_{6})+k_{4}k_{6}
\\
\sigma_3=&k_{1}k_{3}k_{5} x_1(x_{1}+x_{2}+x_{5})+k_{1} ( k_{3}k_{4}+k_{3}k_{6}+k_{4}k_{6} ) x_{1}\\ & \quad+k_{1}k_{6} ( k_{3}+k_{4} ) x_{2}+k_{6} ( k_{2}+k_{3} ) (k_{5}x_{5}+k_{4}).
\end{align*}
{In order to verify that all eigenvalues have negative real part we use the Hurwitz conditions for polynomials of degree three:
\begin{align*}
 \sigma_1>0,\quad\sigma_3>0,\quad\sigma_1\cdot \sigma_2-\sigma_3>0.
\end{align*}
The three expressions on the left side of each inequality are polynomials in the parameters and $x$ with all coefficients positive, and hence are positive when evaluated at positive values of $k$ and $x$.}
 \end{example}

\section*{Appendix}
For convenient reference, we record a basically known result.
\begin{lemma}\label{linalglem}
Let $0<s<n$ and $A\in \mathbb R^{n\times s}$, $B\in \mathbb R^{s\times n}$ such that 
\(
{\rm rank}\,AB=s\text{   and   } (AB)^2=AB.
\)
Then $BA=I_s$.
\end{lemma}
\begin{proof}
Recall that $\mathbb R^n$ is the direct sum of the eigenspaces of $AB$ for eigenvalues $1$ and $0$. Let $z_1,\ldots,z_s$ be a basis of the former. Then $Bz_1,\ldots,Bz_s$ are linearly independent due to 
\[
\sum \lambda_i Bz_i=0\Rightarrow 0=\sum \lambda_i ABz_i=\sum\lambda_iz_i,
\]
hence they form a basis of $\mathbb R^s$. Finally, $ABz_i=z_i$ implies $(BA)Bz_i=Bz_i$ for $1\leq i\leq s$.
\end{proof}

\noindent{\bf Acknowledgements.} The work of NK and SW has been supported by the bilateral project ANR-17-CE40-0036 and DFG-391322026 SYMBIONT.  EF has been supported by the Independent Research Fund of Denmark. SW thanks the Department of Mathematical Sciences of the University of Copenhagen  for its hospitality during a research visit when essential parts of the present manuscript were devised. Likewise, EF thanks the hospitality of RWTH Aachen where this project was initiated.

\end{document}